\documentclass[a4paper,11pt,reqno]{amsart}               
\usepackage[T1]{fontenc}                                                                       
\usepackage[utf8]{inputenc}                                                                  
\usepackage{comment}                                                                         
\usepackage{mparhack}                                                                         
\usepackage{amsmath,amssymb,amsthm,mathrsfs,eucal}                                                         
\usepackage{booktabs}                                                                          
\usepackage{graphicx,subfig}  
\usepackage{wrapfig}                                                                              
\usepackage[bookmarks=true,colorlinks=true]{hyperref}                      
\usepackage{bm}

\newtheorem{teo}{Theorem}[section]

\title
[Global and Local stability]{Global and Local stability for a non-linear hyperbolic system model for the role of stem cells in physiological homeostasis}

\author{Laura Di Bernardo}
\address{Laura Di Bernardo - DISIM - Department of Information Engineering, Computer Science and Mathematics, University of L'Aquila, Via Vetoio 1 (Coppito)
67100 L'Aquila (AQ) - Italy}
\email{laura.dibernardo@univaq.it}

\author{Donatella Donatelli}
\address{Donatella Donatelli- DISIM - Department of Information Engineering, Computer Science and Mathematics, University of L'Aquila, Via Vetoio 1 (Coppito)
67100 L'Aquila (AQ) - Italy}
\email{donatella.donatelli@univaq.it}

\date{}

\begin{document}

\begin{abstract}
In this paper we propose an existence and uniqueness theory for the solutions of a system of non-linear hyperbolic conservation laws, structured in age and maturity variables, representing a tissue environment. In particular we are interested in the investigation of the role of stem cells in its homeostasis. The main result presented in this paper is the consistence of the stability results arising from the analisys of the model  we designed with the experimental observations on which several branches of medicine are currently attempting to trade on their research activity.
\end{abstract}


\maketitle

\section{Introduction}
A good population model has to reproduce satisfactorily behaviors observed by Biologists. A significant class of models which are able to fulfill this requirements are the structured one. They describe the distribution of individuals through different classes, determined by individual differences related to decisive factors of the dynamics we are interested in.
This type of models have the advantage to be able to make a connection between the individual and the population level so that they are able to catch dynamical behaviors which other models can not.

The first authors that attempted to construct a structured model, as a general frame able to describe the intrinsic non linearity of real biological systems, where Gurtin and MacCamy in \cite{rif4}. Their work, afterwards, has inspired the rise of a wide literature; see \cite{rif7}, \cite{rif3}, \cite{rif2},  \cite{rif5}, \cite{rif6}, \cite{rif1}.

In this paper we design a model able to confirm the experimental evidence that stem cells are able to maintain and eventually recover the homeostasis of both animal and vegetal tissues.

In order to do this, our system has been outlined in three groups of cells cohabiting and co-existing, has follows:
\begin{itemize}
\item \emph{Stem cells} that are undifferentiated cells able both to self-convert into specialized units or just to divide depending on the external signals.
\item \emph{Proliferating cells} that are the healthy building blocks of  the tissue as well as the headquarters of its specific physiological functions.
\item \emph{Damaged cells} which are cells that, due to an error in their protein synthesis traffic, caused by alterations in their genetic patrimony, show an impaired physiological activity. In particular, in this paper, they are interpreted as potentially cancerous cells.
\end{itemize}
To describe the dynamical behavior of the first two groups it has been used a model given by an age and maturity-structured system of hyperbolic conservation laws, widely used by M.C. Mackey, M. Adimy and R. Rudnicki (see \cite{rif10}, \cite{rif8}, \cite{rif9},  \cite{rif11}) of the type:
\begin{eqnarray}
\label{model}
z_a(m,a,t)+\left(u(m)z(m,a,t)\right)_m+z_t(m,a,t)&=&f(t,m,Z(m,t))
\end{eqnarray}
\begin{eqnarray}
\label{bc}
Z(m,t)&=&\int_{0}^{\tau}z(m,a,t)da,
\end{eqnarray}
with suitable initial and boundary conditions, where $z(m,a,t)$ is the density of the population of age \emph{a} and maturity \emph{m} at time \emph{t}, $Z(m,t)$ represents the total population over all the age ranges, $u(m)$ is its velocity of maturation, while $f(t,m,Z(m,t))$ represents the cell flux in and out for each group.\\
Here, in order to describe the third group of cells, instead, we consider more appropriate to use a structured partial differential equation of Gompertzian type (see \cite{rif12}) in which the term $f(t,m,Z(m,t))$ assumes the non-linear form
\begin{displaymath}
f(t,m,Z(m,t))=z-z\ln(z).
\end{displaymath}
In the papers \cite{rif10}, \cite{rif8}, \cite{rif9},  \cite{rif11}  the cells are divided in only two groups: proliferating and resting ones; both of them are allowed to be stem cells when their maturity variable is set at zero level. In particular, in $[10]$ stability results related to cells of maturity $m$ are recovered by the stability of the system at the starting point $m=0$.\\
In our paper we extend the existence, uniqueness and stability analysis of a system of equations of the type \eqref{model}-\eqref{bc} studied in \cite{rif10}, \cite{rif8}, \cite{rif9},  \cite{rif11}, to the system provided with damaged cells, represented by the Gompertzian equation.

In our model it has no biological sense to analyze the system at $m=0$ since the dynamics of this group of cells is represented already by a dedicated equation of the type \eqref{model}-\eqref{bc}.

This plan of the paper is as follows. In Section 2 we present our model. In  Section 3, by means of the variation of constants formula, we provide a solution for any equations. The Section $4$ is devoted to the study of the existence and uniqueness of solutions for our model and in  Section 5 we study its stability. In Section 6 we perform some numerical simulations in order to confirm the stability result. We end the paper with some comments that can be found in Section 7.
\section{The model}
In order to fix the model analyzed in this paper we assume that
\begin{itemize}
\item[\bf{A1.}] Cellular proliferation proceeds simultaneously with cellular maturation.
\item[\bf{A2.}] If $\emph{m}$ is the maturation of the mother cell at the moment it starts to divide, the generated daughter cells will have a maturation $g(m)$, where g is a strictly increasing continuous function such that  $g(m)\leq m$. 
\item[\bf{A3.}] In order to consider a kind of damages that stress aggression and speed of cell reproduction, the velocity of maturation of healthy cells $v(m)$ is supposed to be different from velocity of maturation $u(m)$ of cancerous cells.
Moreover, the speeds of cell maturation, are functions $u,v:[0,m_F]\rightarrow[0,\infty)$ continuously differentiable and such that $v(0)=v(m_F)=0$, $u(0)=u(m_F)=0$ and $v(m)>0$, $u(m)>0$ for all $ m \in (0,m_F)$.
\item[\bf{A4.}] In order to stress the aggressiveness of damaged cells, we set furthermore the function $\sigma(m)$, that represent the physiologic rate of loss related to this group, such that $0<\sigma(m)<< 1$.
\end{itemize}
\vspace{3mm}
We give now an accurate description of each group of cells which constitute our system.
\subsection{The stem group} 
We denote by $n(t,m,a)$ the density of undifferentiated units, that are in a resting phase waiting for an external signal. They are supposed to be lost with a random rate $\delta(m)$ (for physiological reasons) and to be introduced in the proliferating phase at a continuous rate $\beta(N(m,t),m)$ (specific of any population). Hence their dynamics is described by the conservation law
\begin{displaymath}
n_a+n_t+\left(n v(m)\right)_m=-\left(\delta(m)+\beta\left(m,N(t,m)\right)\right)n
\end{displaymath}
with

{\bf Initial condition:}
$n(0,m,a)=\phi(a,m)$, for any $(m,a)\in[0,1]\times[0,\infty)$ where $\phi \in C\left([0,1]\times[0,\infty)\right)$ and
\begin{displaymath}
\lim_{a\rightarrow \infty}\phi(a,m)=0.
\end{displaymath}

{\bf Boundary condition:}
\begin{displaymath}
n(t,m,0)= \begin{cases}2(g^{-1})(m)p(t,(g^{-1})(m),\tau) & \textrm{if $t>\overline{\tau}$}\\
		0 & \textrm{if $t<\overline{\tau}$}
		\end{cases}
\end{displaymath}
that represents the flux of daughter cells into the resting phase.

The functions  $\delta(m)$ and $\beta(N(m,t),m)$ are supposed to be positive and continuous.\\
The total density of stem cells is assumed to be 
\begin{displaymath}
N(t,m)=\int_{0}^{\infty}n(t,m,a)da.
\end{displaymath}
\subsection{The proliferating group}
We denote by $p(t,m,a)$ the density of cells that go on through their cell cycle and carry out properly the physiological activity of the tissue. They are supposed to have age \emph{a} and maturity \emph{m} that proceed respectively from $(a,m)=(0,0)$ (that represents the moment in which the daughter cell is generated) to $(a,m)=(\overline{\tau},m_F)$ (which represents the time at which, in turn, the same unit will divide to give rise to the next generation of daughter cells).\\
Moreover, at this phase, they can be lost randomly at an age independent rate $\gamma(m)$ supposed to be a continuous and positive function. So the conservation law for the these cells is given by
 \begin{displaymath}
p_a+p_t+\left(p v(m)\right)_m=-\gamma(m)p
\end{displaymath}
with

{\bf Initial condition:}
$p(0,m,a)=\psi(a,m)$ for any $(m,a)\in[0,1)\times[0,\overline{\tau})$ where $\psi \in C\left(\left[0,1\right]\times\left[0,\overline{\tau}\right]\right)$

{\bf Boundary condition:}
$p(t,m,0)=\beta\left(m, N(t,m)\right)N(t,m)$ that represents the flux of resting cells into the active group.

The total density of proliferating cells is 
\begin{displaymath}
P(t,m)=\int_{0}^{\overline{\tau}}p(t,m,a)da.
\end{displaymath}
\subsection{Damaged cancerous type cells}
Finally,  we denote by $c(t,m,a)$ the density of cells that instead go on through their cell cycle with a damaged physiological activity. They are supposed to be lost physiologically with a rate $\sigma(m)$ and to be converted from the proliferating phase with a rate $\alpha\left(m, P(m,t)\right)$ (specifically related to the kind of damage occurred in their physiological activity).

Also in this case, both $\alpha\left(m, P(m,t)\right)$ and $\sigma(m)$, are supposed to be continuous and positive functions.

Moreover, any damage can occur in the cell from the age $a=\underline{\tau}$ (supposed to be the moment in which any cell start the replication of its genetic patrimony) to the age $a=\overline{\tau}$ (that, as before, is the moment in which they divide).\\
Their dynamics is described by the following Gompertzian PDE,
\begin{displaymath}
c_a+c_t+\left(c u(m)\right)_m=\alpha\left(m, P(m,t)\right)c-\sigma(m)c\ln(c)
\end{displaymath}
with

{\bf Initial condition:}
$c(0,m,a)=\Omega(a,m)$ for any $(m,a)\in[0,m_F)\times[\underline{\tau},\overline{\tau})$, where $\Omega \in C\left(\left[0,1\right]\times\left[\underline{\tau},\overline{\tau}\right]\right)$.
Moreover, for this group of cells, consistently with the genetic damage processes that cause the apparition of cancerous cells, we make the following additional assumption: 
\begin{equation}
\label{A5}
c(0,m,a)<1 
\end{equation}

{\bf Boundary condition:}
$$c(t,m,\underline{\tau})=\alpha\left(m,P(m,t)\right)P(t,m)$$
 that represents the flux of resting cells into the active group.

Let us observe that, in order to be consistent with the Biology of the process, we suppose that the damage process can't occur in the cell dynamics before a certain age $\underline{\tau}$.

The total density of the damaged cells is
\begin{displaymath}
 C(t,m)=\int_{\underline{\tau}}^{\overline{\tau}}c(t,m,a)da.
\end{displaymath}

\vspace{3mm}
Finally, we obtain the following model
\begin{eqnarray}
\label{eqp}
		p_a+p_t+\left(p v(m)\right)_m&=&-\gamma(m)p,\\
\label{eqn}
		n_a+n_t+\left(n v(m)\right)_m&=&-\left(\delta(m)+\beta\left(m,N(t,m)\right)\right)n,\\
\label{eqc}
		c_a+c_t+\left(c u(m)\right)_m&=&\alpha\left(m, P(m,t)\right)c-\sigma(m)c\ln(c),\\
\label{totp}
		P(t,m)&=&\int_{0}^{\overline{\tau}}p(t,m,a)da,\\
\label{totn}
		N(t,m)&=&\int_{0}^{\infty}n(t,m,a)da,\\
\label{totc}
		C(t,m)&=&\int_{\underline{\tau}}^{\overline{\tau}}c(t,m,a)da,
		\end{eqnarray}
		with the boundary conditions
		\begin{eqnarray}
	\label{bcp}
		p(t,m,0)&=&\beta\left(m, N(t,m)\right)N(t,m),\\
	\label{bcn}
		n(t,m,0)&=& \begin{cases}2(g^{-1})(m)p(t,(g^{-1})(m),\tau) & \textrm{if $t>\overline{\tau}$}\\
		0 & \textrm{if $t<\overline{\tau}$},
		\end{cases}\\
	\label{bcc}
		c(t,m,\underline{\tau})&=&\alpha\left(m, P(m,t)\right)P(t,m),
		\end{eqnarray}
		and the initial conditions
		\begin{eqnarray}
		\label{inp}
		p(0,m,a)&=&\psi(a,m),\\
		\label{inn}
		n(0,m,a)&=&\phi(a,m),\\
		\label{inc}
		c(0,m,a)&=&\Omega(a,m).
		\end{eqnarray}
		\section{The solution}
		Since in this paper we want to study the stability of the total population of cells, here we recover the evolution equations  for $N(m,t)$, $P(m,t)$ and $C(m,t)$. The main tools will be the use of the method of characteristics and of the variation of constants formula.\\
		We start by the analysis of the equations \eqref{eqp},\eqref{totp}, \eqref{bcp} and \eqref{inp} namely
		\begin{eqnarray*}
		p_a+p_t+\left(p v(m)\right)_m&=&-\gamma(m)p\\
		P(t,m)&=&\int_{0}^{\overline{\tau}}p(t,m,a)da\\
		p(t,m,0)&=&\beta\left(m,N(t,m)\right)N(t,m)\\
		p(0,m,a)&=&\psi(a,m).
		\end{eqnarray*}
We define the characteristic system as follows
\begin{eqnarray}
\label{carp}
\begin{cases}
a'(s)=1 \\
t'(s)=1 \\
m'(s)=v(m)\\
p'\left(a(s),t(s),m(s)\right)=-\left(\gamma(m)+v'(m)\right)p\left(a(s),t(s),m(s)\right).
\end{cases}
\end{eqnarray}
In order to avoid the difficulty in solving the last equation of the system \eqref{carp}, due to its strong non-linearity, we define  the characteristic curve $s \rightarrow \pi_s(m)$ through $(0,m)$ with $m\in[0,1]$ as a solution of 
\begin{eqnarray*}
\frac{d\pi_s(m)}{ds}(s)&=&v\left(\pi_s(m)\right) \quad s\leq0\\
\pi_{0}(m)&=&m
\end{eqnarray*}
such that $\pi_{s}(0)=0$.\\
 Then, the solution of \eqref{eqp}, \eqref{totp}, \eqref{bcp} and \eqref{inp}  is given by
	\begin{eqnarray}
	\label{solp}
	p=\begin{cases}
	\varsigma(m,t) \Gamma\left(\pi_{-t}(m);a-t\right) & \textrm{if $a>t$}\\
	\varsigma(m,t)N\left(\pi_{-a}(m), t-a \right)\beta\left(N\left(\pi_{-a}(m), t-a\right),\pi_{-a}(m)\right) & \textrm{if $a<t$}
	\end{cases}
	\end{eqnarray}
	with $\varsigma(m,t)=exp\left(-\int_{0}^{t}\gamma\left(\pi_{-s}(m)\right)+v'\left(\pi_{-s}(m)\right)ds\right)$.\\
By integrating  \eqref{eqp} over the age variable in $[0,\overline{\tau}]$ one obtains
\begin{eqnarray}
\label{Psol1}
P_t(m,t)+\left(P(m,t)v(m)\right)_m+p(\overline{\tau},m,t)-p(0,m,t)=-\gamma(m)P(t,m).
\end{eqnarray}
By using the boundary condition \eqref{bcp} and the solution \emph{p} computed in \eqref{solp}, we can  write \eqref{Psol1} as
	\begin{eqnarray}
	\label{eqP}
	P_t \!\!\!\!\!&+&\!\!\!\!\!(P(t,m)v(m))_m=\gamma(m)P(t,m)+ \beta\left( N(t,m),m\right)N(t,m) \\
	\!\!\!\!\!&-&\!\!\!\!\!\begin{cases}\varsigma(m,t) \Gamma\left(\pi_{-t}(m);\overline{\tau}-t\right) & \textrm{if $a>t$}\\
	\varsigma(m,t) N\left(\pi_{-t}(m), t-\overline{\tau} \right)\beta\left( N\left(\pi_{-\overline{\tau}}(m),t-\overline{\tau}\right),\pi_{-\overline{\tau}}(m)\right) & \textrm{if $a<t$} \nonumber.
	\end{cases}
	\end{eqnarray}
From which we get, first by integrating over the variable $m$, and then by  using the variation of constants formula with initial datum  $P(t,m)=\psi(m,t)$,
\begin{eqnarray}
\label{Psol}
P(t,m)\!=\!e^{-(v'(m)-\gamma(m))t}\!\left(\!\psi(m,t)\!\!+\!\!\int_{0}^{t}\!\!\!e^{(v'(m)-\gamma(m))s}\!\!\!\int_{0}^{m}\!\!\!F(N(s,z),z)dzds\right)
\end{eqnarray}
where 
\begin{eqnarray*}
&F&\!\!\!\!\! \left(N(m,t),m\right)=\beta\left(N(m,t),m\right)N(m,t)\\
&-&\!\!\!\!\! \begin{cases}\varsigma(m,t) \Gamma\left(\pi_{-t}(m);\overline{\tau}-t\right) & \textrm{if $a>t$}\\
	\varsigma(m,t) N\left(\pi_{-t}(m), t-\overline{\tau} \right)\beta\left( N\left(\pi_{-\overline{\tau}}(m),t-\overline{\tau}\right),\pi_{-\overline{\tau}}(m)\right) & \textrm{if $a<t$} \nonumber.
	\end{cases}
\end{eqnarray*}
	From \eqref{eqP} we can see that, in order to obtain the explicit solution  for $P(t,m)$ we need to recover the solution for $N(t,m)$.\\
As before, by integration over the age variable the equation \eqref{eqn}, we obtain
\begin{eqnarray}
\label{eqN}
N_t+\left(N(t,m)v(m)\right)_m=-\left(\delta(m)+\beta\left(m,N(t,m)\right)\right)N(t,m)\nonumber \\
+\begin{cases}2(g^{-1})(m)p(t,(g^{-1})(m),\tau) & \textrm{if $t>\overline{\tau}$}\\
		0 & \textrm{if $t<\overline{\tau}$}.
		\end{cases}
\end{eqnarray}
	By plugging \eqref{solp} in \eqref{eqN} and by applying again the method of characteristics we obtain the following solution for $N$,
	\begin{eqnarray}
	\label{Nsol}
	N(t,m)&=&\phi \left(\overline{\tau},\pi_{-(t-\overline{\tau})}(m)\right)K(t-\overline{\tau},m)\nonumber \\
	&-&\int_{\overline{\tau}}^{t}K(t-s,m)\beta\left[N(s,\pi_{-(t-s)}(m)),m\right]N(s,\pi_{-(t-s)}(m))ds\nonumber \\
	&+& \int_{\overline{\tau}}^{t}K(t-s,m)f_2\left(\pi_{-(t-s)}(m),m\right)ds\nonumber \\
	&\times& \beta\left[N(s-\overline{\tau},g^{-1}\left(\pi_{-(t-s)}(m)),m\right)\right] \nonumber \\
	&\times& N\left(s-\overline{\tau},g^{-1}\left(\pi_{-(t-s)}(m)\right)\right)
	\end{eqnarray}
	with initial datum $N(t,m)=\phi(t,m)$ where, $\phi \in C\left([0,\overline{\tau}]\times[0,g(1)]\right)$ and
	\begin{eqnarray*}
	f_2(m)&=&2(g^{-1})'(m)\varsigma \left(g^{-1}(m),t\right)\\
	K(m,t)&=&\exp\left(-\int_{0}^{t}\delta\left(\pi_{-\sigma}(m)\right)+v'\left(\pi_{-\sigma}(m)\right)d\sigma\right).
	\end{eqnarray*}
Similarly, we can associate to the problem \eqref{eqc}, \eqref{totc}, \eqref{bcc} and \eqref{inc} the following characteristic system 
\begin{eqnarray}
\begin{cases}
a'(s)=1\\
t'(s)=1\\
m'(s)=u(m)\\
c'\left(a(s),t(s),m(s)\right)=-u'(m)c\left(a(s),t(s),m(s)\right)
\\\hspace{3cm}+\alpha(P(m,t),m)c\left(a(s),t(s),m(s)\right)\\
\hspace{3cm}-\sigma(m(s))c\left(a(s),t(s),m(s)\right)\ln\left(c\left(a(s),t(s),m(s)\right)\right).
\end{cases}
\end{eqnarray}
By using the method of characteristic and the standard solution for Gompertzian equations we get
	\begin{eqnarray}
	\label{solc}
	c(t,m,a)\!\!=\!\!\exp\!\!{\left[ \frac{\widetilde{\alpha}(m, P(m,t))}{\sigma(m)}\!\!-\!\!\left(\frac{\widetilde{\alpha}(m, P(m,t))}{\sigma(m)}\!\!-\!\!\ln{(\Omega(m,a))}\!\!\right)e^{-\sigma(m)t}\!\right]} 
	\end{eqnarray}
	with 
	\begin{equation}
	\label{alpha}
	\widetilde{\alpha} \left(m,P(t,m)\right)=u'(m)-\alpha\left(m,P(t,m)\right).
	\end{equation}
	Finally, by using \eqref{solc} and integrating \eqref{eqc} over the age variable in $a\in[\underline{\tau},\overline{\tau}]$ we get
	\begin{eqnarray}
	\label{eqC}
	C_t\!\!\!\!\!&+&\!\!\!\!\!(C(m,t)u(m))_m=\alpha(P(m,t),m)C(m,t)+\alpha(P(m,t),m)P(m,t)\nonumber \\
	\!\!\!\!\!&-&\!\!\!\!\!\sigma(m)\int_{\underline{\tau}}^{\overline{\tau}}\!\!c\log(c)da \nonumber\\
	\!\!\!\!\!&-&\!\!\!\!\!\exp\!\!\left(\frac{\widetilde{\alpha}(P(m,t),m)}{\sigma(m)}-\left(\frac{\widetilde{\alpha}(P(t,m),m)}{\sigma(m)}-\ln{(\Omega(m,a))}\!\!\!\right)e^{-\sigma(m)t}\right),
	\end{eqnarray}
	from which we have
	\begin{eqnarray}
	\label{Csol}
C(t,m)&=&\Omega(m,a)\Big[e^{-\int_{0}^{t}(u'(m)-\alpha(P(m,s),m))ds}\nonumber\\
&+&\int_{0}^{t}e^{\int_{t}^{s}(u'(m)-\alpha(P(m,r),m))dr}\int_{0}^{m}F(P(s,z),z)dzds\Big]
\end{eqnarray}
 where
\begin{displaymath}
F((P(m,t),m)=\alpha(P(m,t),m)P(m,t)-F_{C}
\end{displaymath}
while
\begin{eqnarray*}
F_{C}&=&-\sigma(m)\int_{\underline{\tau}}^{\overline{\tau}}c\log(c)da\\  
&-&exp\left(\frac{\widetilde{\alpha}(P(m,t),m)}{\sigma(m)}-\left(\frac{\widetilde{\alpha}(P(t,m),m)}{\sigma(m)}-\ln(\Omega(m,a))\right)e^{-\sigma(m)t}\right)
\end{eqnarray*}
Similarly, as before, we note that the explicit solution for $C$ is a function of $P$ and $N$.
\section{Existence and uniqueness}
In this section we focus on the existence of solutions for \emph{N}, \emph{P} and \emph{C}.\\
We start with \emph{N} and we rewrite \eqref{eqN} in the following compact form
\begin{eqnarray}
N_t+v(m)N_m=G\left(m,\overline{N}\right)N+F_N(m,t),
\end{eqnarray}
where 
\begin{displaymath}
G(m,\overline{N})=-\left[\delta(m)+\beta\left(m, \overline{N}\right)+v'(m)\right]
\end{displaymath}
and
\begin{displaymath} 
\overline{N}=\int_{0}^{m_F}N(t,m)dm
\end{displaymath}
represents the total number of all stem cells over all the maturity range, while
\begin{displaymath}
F_N(m,t)=\begin{cases}2(g^{-1})(m)p(t,(g^{-1})(m),\tau) & \textrm{if $t>\overline{\tau}$}\\
		0 & \textrm{if $t<\overline{\tau}$}.
		\end{cases}
\end{displaymath}
As a consequence \eqref{Nsol} assumes the form
\begin{eqnarray}
N(t,m)&=&\int_{0}^{t} F_N(m,t)\exp\left(\int_{r}^{t}G(\pi_{t-s}m,\overline{N}(s)ds)dr\right) \nonumber \\
&+& N \left(0,\pi_{-t}m\right)\exp\left(\int_{0}^{t}\exp\int_{0}^{t}G(\pi_{t-s}m,\overline{N}(s)ds)\right).
\end{eqnarray}
Now, on the Banach space $C[0,\tau]$ we define the operator $\mathbb{B}$ as follows:\\
$\mathbb{B}:C[0,\tau]\rightarrow [0,\tau]$ such that
\begin{eqnarray}
\mathbb{B}\overline{N}(t)=\int_{0}^{m_F}N(m,t)dm
\end{eqnarray}
 with norm 
\begin{eqnarray}
\label{nrmN}
||N||=e^{-\lambda t}\max_{0\leq t\leq\tau}|N|
\end{eqnarray}
where $\lambda>0$.\\
We want to show that $\mathbb{B}$ is a contractive operator. First of all we observe that by definition of the norm \eqref{nrmN} for any $ s \in[0,\tau]$ we have
\begin{displaymath}
|N_1(s)-N_2(s)|\leq e^{\lambda s} ||N_1-N_2||.
	\end{displaymath}
Since, by their structure, $G$ and its derivative $G_N$ are bounded from above by two positive constants $k_1$ and $k_2$ respectively, we get 
\begin{eqnarray*}
\Big| exp\int_{r}^{t}\left(G(\pi_{t-s}m,\overline{N_1}(s)ds\right)-exp\int_{r}^{t}\left(G(\pi_{t-s}m,\overline{N_2}(s)ds\right) \Big| \\
\leq e^{k_{1}(t-r)}\Big| \int_{r}^{t}G(\pi_{t-s}m,\overline{N_1}(s))ds-\int_{r}^{t}G(\pi_{t-s}m,\overline{N_2}(s))ds \Big|\\
\leq e^{k_{1}t}\int_{r}^{t}k_2 \big| N_1(s)-N_2(s) \big| ds \leq  e^{k_{1}t}\int_{r}^{t}k_2 e^{\lambda s}  ||N_1-N_2||ds\\
\leq \frac{1}{\lambda} k_2 e^{k_{1}t} e^{\lambda t} ||N_1-N_2||.
\end{eqnarray*}
 Due to the initial hypothesis  $\bf{A2}$ on $g(m)$, both $F_N(m,t)$ and $N(m,0)$ are bounded,
  $$|F_N(m,t)|<k_3 \qquad |N(m,0)|<k_4.$$\\
	So it follows that for any $r \in[0,t]$
\begin{eqnarray*}
||\mathbb{B}N_1(t)-\mathbb{B}N_2(t)||\!\!\!&\leq&\!\!\! \frac{1}{\lambda} k_2 e^{k_{1}t} e^{\lambda t} ||N_1-N_2||\left(\int_{0}^{m_F}\!\!\!\!\! \int_{0}^{t}k_3dr dm+\int_{0}^{m_F}\!\!\!\!\!k_4 dm\right)\\
\!\!\!&\leq& \!\!\!\frac{1}{\lambda} k_2 e^{k_{1}t} e^{\lambda t} m_{F}\left(k_{3}t+k_4\right)||N_1-N_2||. 
\end{eqnarray*}
If we choose $\lambda$ sufficiently large, the operator $\mathbb{B}$ is contractive so it admits a unique fixed point $\overline{N}$ such that $\mathbb{B}\overline{N}=\overline{N}$. Consequently, for any given initial condition, the equation \eqref{eqN} has exactly one solution for any $t\in[0,\tau]$. Then by method of steps is possible to iterate the proof for any $ t \geq 0$.

To show the existence and uniqueness for $P$ we use exactly the same procedure taking into account the appropriate initial conditions.\\
Analogously to prove it for $C$ we rewrite \eqref{eqC} in the compact form 
\begin{eqnarray}
C_t+u(m)C_m=H\left(m,\overline{C}\right)C+F_C(m,t),
\end{eqnarray}
where 
\begin{displaymath}
H(m,\overline{C})=-\left[u'(m)-\alpha(P(m,t),m)\right]
\end{displaymath}
and
\begin{displaymath} 
\overline{C}=\int_{0}^{m_F}C(t,m)dm
\end{displaymath}
represents the total number of all damaged cells over all the maturity range, while
\begin{eqnarray*}
F_C(m,t)&=&\alpha(P(m,t),m)P(m,t)-\sigma(m)\int_{\underline{\tau}}^{\overline{\tau}}\!\!c\log(c)da\\
	\!&-&\!exp\left(\frac{\widetilde{\alpha}(P(m,t),m)}{\sigma(m)}-\left(\frac{\widetilde{\alpha}(P(t,m),m)}{\sigma(m)}-\ln(c_0)\right)e^{-\sigma(m)t}\right).
\end{eqnarray*}
As a consequence \eqref{Csol} assumes the form
\begin{eqnarray}
C(t,m)&=&\int_{0}^{t}F_C(m,t)\exp\left(\int_{r}^{t}H(\pi_{t-s}m,\overline{C}(s)ds)dr\right) \nonumber \\
&+& C \left(0,\pi_{-t}m\right)\exp\left(\int_{0}^{t}\exp\left(\int_{0}^{t}H(\pi_{t-s}m,\overline{C}(s)ds)\right)\right).
\end{eqnarray}
Now, on the Banach space $C[0,\tau]$ we define the operator $\mathbb{A}$ as follows:\\
$\mathbb{D}:C[0,\tau]\rightarrow [0,\tau]$ such that
\begin{eqnarray}
\mathbb{D}\overline{C}(t)=\int_{0}^{m_F}C(m,t)dm
\end{eqnarray}
 with norm 
\begin{eqnarray}
\label{nrmC}
||C||=e^{-\lambda t}\max_{0\leq t\leq\tau}|C|
\end{eqnarray}
where $\lambda>0$.\\
We want to show that $\mathbb{D}$ is a contractive operator. First of all we observe that by definition of the norm \eqref{nrmC} for any $ s \in[0,\tau]$ we have
\begin{displaymath}
|C_1(s)-C_2(s)|\leq e^{\lambda s} ||C_1-C_2||.
	\end{displaymath}
Since, by their structure, $H$ and its derivative $H_C$ are bounded from above by two positive constants $\tilde k_1$ and $\tilde k_2$ respectively, we get 
\begin{eqnarray*}
\Big| \exp\int_{r}^{t}\left(H(\pi_{t-s}m,\overline{C_1}(s)ds\right)-\exp\int_{r}^{t}\left(H(\pi_{t-s}m,\overline{C_2}(s)ds\right) \Big| \\
\leq e^{\tilde k_{1}(t-r)}\Big| \int_{r}^{t}H(\pi_{t-s}m,\overline{C_1}(s))ds-\int_{r}^{t}H(\pi_{t-s}m,\overline{C_2}(s))ds \Big|\\
\leq e^{\tilde k_{1}t}\int_{r}^{t}k_2 \big| C_1(s)-C_2(s) \big| ds \leq  e^{\tilde k_{1}t}\int_{r}^{t}k_2 e^{\lambda s}  ||C_1-C_2||ds\\
\leq\frac{1}{\lambda} \tilde k_2 e^{\tilde k_{1}t} e^{\lambda t} ||C_1-C_2||.
\end{eqnarray*}
 Due to the initial hypothesis of continuity an positiveness made on $\alpha$ and on the previous results on $P$, both $F_C(m,t)$ and $C(m,0)$ are bounded,

$$|F_C(m,t)|<\tilde k_3 \qquad |C(m,0)|<\tilde k_4.$$\\	
So it follows that for any $r \in[0,t]$
\begin{eqnarray*}
||\mathbb{D}C_1(t)-\mathbb{D}C_2(t)||\!\!\!\!\!&\leq&\!\!\!\!\! \frac{1}{\lambda} \tilde k_2 e^{\tilde k_{1}t} e^{\lambda t} ||C_1-C_2||\left(\int_{0}^{m_F}\!\!\!\int_{0}^{t}\tilde k_3dr dm+\int_{0}^{m_F}\!\!\!\tilde k_4 dm\right)\\
\!\!\!\!\!&\leq& \!\!\!\!\!\frac{1}{\lambda} \tilde k_2 e^{\tilde k_{1}t} e^{\lambda t} m_{F}\left(\tilde k_{3}t+\tilde k_4\right)||C_1-C_2||. 
\end{eqnarray*}
If we choose $\lambda$ sufficiently large, the operator $\mathbb{D}$ is contractive so it admits a unique fixed point $\overline{C}$ such that $\mathbb{D}\overline{C}=\overline{C}$. Consequently, for any given initial condition, the equation \eqref{eqC} has exactly one solution for any $t\in[0,\tau]$. Then by method of steps is possible to iterate the proof for any $ t \geq 0$. Finally we have proved the following global existence result. 
\begin{teo}[Global existence and uniqueness]
Let us assume that the hypotheses $\bf{A1}$- $\bf{A4}$ hold, then the system \eqref{eqp}-\eqref{inc} with the initial conditions \eqref{inp}-\eqref{inc} and the boundary conditions \eqref{bcp}-\eqref{bcc}, has a unique global solution for any $t \geq 0$.
\end{teo}
\section{Stability results and asymptotic behavior}
This section is devoted to prove the stability of solutions for the system \eqref{eqp}-\eqref{inc}. In particular we will prove the local  and global exponential stability of the trivial solution $(N,P,C)\equiv(0,0,0)$ by using an iterative procedure.\\
Let us consider the domain $D=\left([0,\overline{\tau}]\times[0,g(1)]\right)$.\\
We denote by $z^{\phi}$ the generic solution $N,P,C$ with initial condition $\phi$.\\
For completeness, we list below the main definitions of stability we are going to use in this section.
\begin{itemize}
\item[\bf{Def 1:}] The solution $\overline{z}$ with initial condition $\overline{\phi}\in D$ is \emph{locally stable} if for all $\epsilon>0$ exists $k_{\epsilon}>0$ such that if $\phi\in D$ and $||\phi-\overline{\phi}||<k_{\epsilon}$ than $|z^{\phi}-\overline{z}|<\epsilon$ for all $(t,m)\in\left([\overline{\tau},\infty)\times[0,g(1)]\right)$.
\item[\bf{Def 2:}] The solution $\overline{z}$ with initial condition $\overline{\phi}\in D$ is \emph{locally exponentially stable} if for all $\epsilon>0$ exist $c,d>0$ such that if $\phi\in D$ and $||\phi-\overline{\phi}||<\epsilon$ then $|z^{\phi}-\overline{z}|<ce^{-d(t-\overline{\tau})}$ for all $(t,m)\in\left([\overline{\tau},\infty)\times[0,g(1)]\right)$.
\item[\bf{Def 3:}] A solution $\overline{z}$ related to the initial condition $\overline{\phi}$ is globally exponentially stable on $D$ if for all $\phi\in D$ there exists positive constants $c>0$ and $d>0$ such that
\begin{displaymath}
\lim_{t\rightarrow\infty}|z^{\phi}-\overline{z}|\leq ce^{-d(t-\overline{\tau})} 
\end{displaymath}
for all $(t,m)\in\left([\overline{\tau},\infty)\times[0,g(1)]\right).$
\end{itemize}
Through this section we assume that:
\begin{itemize}
\item [\bf{B1.}] The map $x\rightarrow x\beta(m,x)$ is positive and Lipschitz continuous in a neighborhood of zero, that is, there exists $\epsilon>0$, and $k_b>0$ such that 
\begin{eqnarray*}
|x\beta(m,x)-y\beta(m,y)|\leq k_b|x-y|
\end{eqnarray*}
for every $|x|\leq \epsilon$ and $|y|\leq \epsilon$.
\item [\bf{B2.}] The map $x \rightarrow \alpha(m,x)$ is Lipschitz and such that $\alpha(0)\!\!=\!\!0$. 
As a consequence the map $x\rightarrow x\alpha(m,x)$ is positive and Lipschitz continuous in a neighborhood of zero, namely there exists $\epsilon>0$, and $k_a>0$ such that 
\begin{eqnarray*}
|x\alpha(m,x)-y\alpha(m,y)|\leq k_a|x-y|
\end{eqnarray*}
for every $|x|\leq \epsilon$ and $|y|\leq \epsilon$.
\end{itemize}

\subsection{Local and global stability for N-cells} In this section we will rewrite the solution \eqref{Nsol} as
	\begin{eqnarray}
	N(t,m)=\begin{cases} N_0(t,m)+G(t,m)+J(t,m) & \textrm{if $t\geq\overline{\tau}$}\\
	\phi(t,m) & \textrm{if $t\in[0,\overline{\tau}]$}
	\end{cases}
	\end{eqnarray}
	where we set
	\begin{eqnarray*}
	N_0(t,m)\!\!\!&=&\!\!\!\phi \left(\overline{\tau},\pi_{-(t-\overline{\tau})}(m)\right)K(t-\overline{\tau},m)\\
	G(t,m)\!\!\!&=&\!\!\!\int_{\overline{\tau}}^{t}K(t-s,m)\beta\left[N(s,\pi_{-(t-s)}(m)),m\right]N(s,\pi_{-(t-s)}(m))ds\\
	J(t,m)\!\!\!&=& \!\!\!\int_{\overline{\tau}}^{t}K(t-s,m)f_2\left(\pi_{-(t-s)}(m),m\right)ds\\
	\!\!\!&\times&\!\!\! \beta\left[N(s-\overline{\tau},g^{-1}\left(\pi_{-(t-s)}(m)),m\right)\right]
	\!\!\times\!\! N\left(s-\overline{\tau},g^{-1}\left(\pi_{-(t-s)}(m)\right)\right)
	\end{eqnarray*}
and we are going to prove the following Theorem
\begin{teo}[Local stability]
\label{T1}
If we suppose $A=\frac{k_b(1+2\varsigma)}{I}<1$, then the trivial solution $N\equiv 0$ of \eqref{eqN} is locally exponentially stable in the sense of definition \bf{Def 2}. 
\end{teo}
\begin{proof}
First of all we define the sequence $(N_n)_{n\in\mathbb{N}}$ as follows
\begin{displaymath}
N_0(t,m)=\begin{cases}
\phi(t,m) \qquad \textrm{if $t\in[0,\overline{\tau}]$}\\
\phi\left(\overline{\tau},\pi_{-(t-\overline{\tau})}(m)\right)k(t-\overline{\tau},m) \qquad \textrm{if $t>\overline{\tau}$}
		\end{cases}
\end{displaymath}

\begin{displaymath}
N_n(t,m)=\begin{cases}
\phi(t,m) & \textrm{if $t\in[0,\overline{\tau}]$}\\
N_0(t,m)+G(N_{n-1})(t,m)+J(N_{n-1})(t,m) &\textrm{if $t>\overline{\tau}$}.
		\end{cases}
\end{displaymath}
Recalling  the definitions of $\pi$ and $g$ we have that $\pi_{-(t-\overline{\tau})}(m)<m<g(1)$ and 
so 
\begin{displaymath}
\phi\left(\overline{\tau},\pi_{-(t-\overline{\tau})}(m)\right)<\epsilon,
\end{displaymath}
moreover
\begin{displaymath}
|K(t-\overline{\tau},m)|\leq e^{-I(t-\overline{\tau})}
\end{displaymath}
with $I=\inf_{m\in[0,g(1)]}(\delta(m)+v'(m))$, so for $n=0$ we get
\begin{displaymath}
|N_0(m,t)|\leq \epsilon e^{-I(t-\overline{\tau})}\leq \epsilon.
\end{displaymath}
If we suppose that $|N_n(m,t)|<\epsilon$ for $n>0$ and by the Lipschitzianity of $\beta$ we obtain the estimates
\begin{eqnarray*}
|G(N_n)(m,t)|&\leq& 2 \epsilon \varsigma k_b \int_{\overline{\tau}}^{t}e^{-I(t-s)}ds,\\
|J(N_n)(m,t)|&\leq& \epsilon  k_b \int_{\overline{\tau}}^{t} e^{-I(t-s)}ds.
\end{eqnarray*}
It follows that
\begin{eqnarray}
N_{n+1}(t,m)=N_0(t,m)+G(N_n)(t,m)+J(N_n)(t,m)
\end{eqnarray}
becomes
\begin{eqnarray*}
|N_{n+1}(m,t)|&\leq& \epsilon e^{-I(t-\overline{\tau})}+2\epsilon \varsigma k_b \int_{\overline{\tau}}^{t}e^{-I(t-s)}ds\\
&+&\epsilon  k_b \int_{\overline{\tau}}^{t}e^{-I(t-s)}ds\\
&\leq&\epsilon \left(e^{-I(t-\overline{\tau})}+\frac{k_b(1+2\varsigma)}{I}(1-e^{-I(t-\overline{\tau})})\right),
\end{eqnarray*}
since $A=\frac{k_b(1+2\varsigma)}{I}<1$, we have that
\begin{eqnarray*}
|N_{n+1}(m,t)|&\leq&\epsilon \left(e^{-I(t-\overline{\tau})}(1-A)+A\right)\\
&\leq& \epsilon \left((1-A)+A\right)\leq \epsilon.
\end{eqnarray*}
By the induction principle
\begin{eqnarray}
\label{prelN}
|N_n(m,t)|\leq\epsilon,
\end{eqnarray}
for all $n\in\mathbb{N}$.
This is a preliminary result to show the exponential stability of our solution.

Let us now define the continuous map $p\in[0,I]\rightarrow\frac{I-p}{1+2\varsigma e^{p\overline{\tau}}}$ such that\\
$k<\frac{I-p}{1+2\varsigma e^{p\overline{\tau}}}<\frac{I}{1+2\varsigma}$. By using \eqref{prelN} and  \textbf{B1} we can prove the following estimates
\begin{eqnarray}
\label{estN}
|N_0(m,t)|&\leq& \epsilon e^{-I(t-\overline{\tau})} \leq \epsilon e^{-p(t-\overline{\tau})} \leq \epsilon\\
\label{estJ}
|J_0(N_0)(m,t)|&\leq& k_b\int_{\overline{\tau}}^{t} |N_0(\pi_{-(t-s)(m)},s)|e^{-I(t-s)}ds \nonumber \\
&\leq&  k_b\int_{\overline{\tau}}^{t} e^{-p(t-\overline{\tau})} e^{-I(t-s)}ds\nonumber\\
&\leq& k_b e^{-It} e^{p\overline{\tau}}\int_{\overline{\tau}}^{t} e^{-(I-p)s}ds\\
\label{estG}
|G_0(N_0)(m,t)|&\leq& 2\varsigma \epsilon k_b e^{-It} e^{2pt} \int_{\overline{\tau}}^{t} e^{-(I-p)s}ds.
\end{eqnarray}
By using \eqref{estN},\eqref{estJ} and \eqref{estG} we get,
\begin{eqnarray*}
|N_1(m,t)-N_0(m,t)|\!\!\!&=&\!\!\!|N_0(m,t)+J(N_0)(m,t)+G_0(m,t)-N_0(m,t)|\\
\!\!\!&=&\!\!\!|J(N_0)(m,t)+G_0(m,t)|\\
\!\!\!&\leq&\!\!\! k_b e^{-It} e^{p\overline{\tau}}\int_{\overline{\tau}}^{t} \!\!\!e^{-(I-p)s}ds+2\varsigma \epsilon k_b e^{-It} e^{2pt} \int_{\overline{\tau}}^{t}\!\!\! e^{-(I-p)s}ds\\
\!\!\!&\leq&\!\!\!k_b e^{-It} e^{p\overline{\tau}}[1+2\varsigma e^{pt}]\int_{\overline{\tau}}^{t}\!\!\! e^{-(I-p)s}ds.
\end{eqnarray*}
Solving the integral $e^{-It} e^{p\overline{\tau}}\int_{\overline{\tau}}^{t} e^{-(I-p)s}ds$ it is possible to complete the estimate as follows
\begin{eqnarray}
|N_1(m,t)-N_0(m,t)| &\leq& \epsilon k_b[1+2\varsigma e^{pt}]\frac{1}{I-p}e^{-p(t-\overline{\tau})}\nonumber \\
&\leq& \epsilon k_b e^{-p(t-\overline{\tau})}.
\end{eqnarray}
By induction and with the same procedure as before we are able to prove that for any $n \in \mathbb{N}$,
\begin{eqnarray}
|N_{n+1}(m,t)-N_{n}(m,t)| \leq \epsilon (k_b)^{n+1} e^{-p(t-\overline{\tau})}.
\end{eqnarray}
\end{proof}

From this result it is now possible to extend our local stability result to a global stability one.
\begin{teo}[Global exponential stability]
\label{T2}
If the hypothesis \textbf{B1} holds, with $k_b$ such that $\frac{k_b(1+2\varsigma)}{I}<1$, then the trivial solution $N\equiv 0$ of \eqref{eqN} is globally exponentially stable in the sense of definition \bf{Def 3}.
\end{teo}
\begin{proof}
It follows from the local stability that
\begin{eqnarray*}
|N^{\phi}(m,t)|\leq||\phi||.
\end{eqnarray*}
So, by using the same arguments of Theorem \ref{T1} we have
\begin{eqnarray}
\lim_{t\rightarrow\infty}|N^{\phi}|=\lim_{t\rightarrow\infty}\epsilon (k_b)^{n+1} e^{-p(t-\overline{\tau})}=0,
\end{eqnarray}
 for all $(t,m)\in\left([\overline{\tau},\infty)\times[0,g(1)]\right)$.
\end{proof}

\subsection{Local and global stability of P-cells}
Let us recall that the solution for $P$ with initial datum $P(t,m)=\psi(m,t)$ where $\psi\in C\left([\underline{\tau},\overline{\tau}]\times[0,g(1)]\right)$ is
\begin{displaymath}
P(t,m)\!=\!e^{-(v'(m)-\gamma(m))t}\!\!\left(\psi(m,t)\!+\!\int_{0}^{t}\!\!\!e^{(v'(m)-\gamma(m))s}\!\!\int_{0}^{m}\!\!\!F((N(s,z),z)dzds\!\!\right)
\end{displaymath}
where 
\begin{eqnarray*}
&F&\!\!\!\!\!\left(N(m,t),m\right)=\beta\left(N(m,t),m\right)N(m,t)\\
&+&\!\!\!\!\!\begin{cases}
	\!\varsigma(m,t) \Gamma\left(\pi_{-t}(m);a-t\right) & \textrm{if $a>t$}\\
	\!\varsigma(m,t)N_{n}\left(\pi_{-a}(m), t-a \right)\beta\!\left(N_{n}\left(\pi_{-a}(m), t-a\right),\pi_{-a}(m)\right)  & \textrm{if $a<t$}
	\end{cases}
\end{eqnarray*}
by using the same procedure of the previous section we define the sequence $(P_n)_{n\in \mathbb{N}}$ as
\begin{eqnarray}
P_n(t,m)&=&\psi(m,t)\Big[e^{-(v'(m)-\gamma(m))t}\nonumber\\
&+&\int_{0}^{t}e^{-(v'(m)-\gamma(m))(t-s)}\int_{0}^{m}F(N_{n}(z,s),z)dzds\Big]
\label{Pn}
\end{eqnarray}
where $N_{n}$ is his in the previous section and 
\begin{eqnarray*}
&F&\!\!\!\!\!\!\left(N_{n}(m,t),m\right)\\
&=&\!\!\begin{cases}
	\!\varsigma(m,t) \Gamma\left(\pi_{-t}(m);a-t\right) & \textrm{if $a>t$}\\
	\!\varsigma(m,t)N_{n}\left(\pi_{-a}(m), t-a \right)\beta\!\left(N_{n}\left(\pi_{-a}(m), t-a\right),\pi_{-a}(m)\right) & \textrm{if $a<t$}
	\end{cases}
\end{eqnarray*}
First of all we estimate the sequence \eqref{Pn} in order to prove a preliminary result of invariance analogous to \eqref{prelN} for $(P_n)_{n\in \mathbb{N}}$ where, as before, we set the initial condition $\psi(m,t)$ such that
 $$\psi(m,t)<\epsilon.$$ \\
By defining $E=\inf{(v'(m)-\gamma(m))}>0$ from \eqref{Pn} we get that
\begin{eqnarray}
\label{prelP}
|P_n(m,t)|\leq\epsilon \left(e^{-Et}+ \tau e^{-E\tau}k_b \right) \leq \overline{K}\epsilon.
\end{eqnarray}
Now, in order to show the stability for $P$, we need to recover some estimates for $|P_{n+1}-P_n|$. For $a>t$ we have
\begin{eqnarray}
\!\!\!|P_{n+1} - P_{n}|=0.
\end{eqnarray}
While for $a<t$ we have
\begin{eqnarray}
|\varsigma(m,t)\beta\left(N_{n+1}(m,t),m\right)N_{n+1}(m,t)\!\!\!\!\!&-&\!\!\!\!\!\varsigma(m,t)\beta\left(N_{n}(m,t),m\right)N_{n}(m,t)|\nonumber\\
\!\!\!\!\!&\leq&\!\!\!\!\! \sup|\varsigma(m,t)|k_b(N_{n+1}-N_{n}).
\end{eqnarray}
Now in both cases, by the global exponential stability proved for  $N$, in Theorem \ref{T2} is possible to conclude the convergence 
\begin{eqnarray}
|P_{n+1}-P_{n}|\rightarrow 0 \quad \textrm{as $n\rightarrow \infty$},
\end{eqnarray}
i.e. the global exponential stability of \emph{P}. Hence we proved the following Theorem.
\begin{teo}[Global exponential stability]
\label{T3}
If the hypothesis \textbf{B1} holds, with $k_b$ such that $\frac{k_b(1+2\varsigma)}{I}<1$, then the trivial solution $P\equiv 0$ of \eqref{Psol} is globally exponentially stable in the sense of definition \bf{Def 3}.
\end{teo}

\subsection{Local and global stability of C-cells}
Finally we apply the same procedure to the last group of cells. Let us recall that the complete solution for  $C$ with initial datum $\Omega(a,m)$ such that $\Omega \in C\left([\underline{\tau},\overline{\tau}]\times[0,g(1)]\right)$ is given by
\begin{eqnarray}
C(m,t)&=&\Omega(a,m)\Big[e^{-\int_{0}^{t}(u'(m)-\alpha(P(m,s),m))ds}\nonumber\\
&+&\int_{0}^{t}e^{\int_{t}^{s}(u'(m)-\alpha(P(m,r),m))dr}\int_{0}^{m}F(P(s,z),z)dzds\Big]
\end{eqnarray}
 where
\begin{displaymath}
F((P(m,t),m)= \alpha(P(m,t),m)P(m,t)-F_{C}
\end{displaymath}
while
\begin{eqnarray*}
F_{C}&=&\sigma(m)\int_{\underline{\tau}}^{\overline{\tau}}c\log(c)da\\  
&+&exp\left(\frac{\widetilde{\alpha}(P(m,t),m)}{\sigma(m)}-\left(\frac{\widetilde{\alpha}(P(t,m),m)}{\sigma(m)}-\ln(\Omega(m,a))\right)e^{-\sigma(m)t}\right)
\end{eqnarray*}
and $\tilde \alpha$ defined as in \eqref{alpha}.

Let us now define the sequence  $(C_n)_{n\in \mathbb{N}}$ as follows
\begin{eqnarray}
C_{n}(m,t)&=&\Omega(a,m)\Big[e^{-\int_{0}^{t}(u'(m)-\alpha(P_{n}(m,s),m))ds}\nonumber\\
&+&\int_{0}^{t}e^{\int_{t}^{s}(u'(m)-\alpha(P_{n}(m,r),m))dr}\int_{0}^{m}F(P_{n}(s,z),z)dzds\Big]
\end{eqnarray}
 where
\begin{displaymath}
F((P_{n}(m,t),m)= \alpha(P_{n}(m,t),m)P_{n}(m,t)-F_{C_n}
\end{displaymath}
while
\begin{eqnarray*}
F_{C_n}&=&\sigma(m)\int_{\underline{\tau}}^{\overline{\tau}}c\log(c)da\\  
&+&\exp\left(\frac{\widetilde{\alpha}(P_{n}(m,t),m)}{\sigma(m)}-\left(\frac{\widetilde{\alpha}(P_{n}(t,m),m)}{\sigma(m)}-\ln(\Omega(a,m))\right)e^{-\sigma(m)t}\right)
\end{eqnarray*}
First of all, we recover a preliminary result analogous to \eqref{prelN} and \eqref{prelP} by estimating $|C_n(m,t)|$.\\
By the definitions of $\pi$ and $g$  and hypothesis \textbf{A3}, we have that
 $\pi_{-(t-\overline{\tau})}(m)<m<g(1)$ so 
\begin{eqnarray}
\Omega\left(\overline{\tau},\pi_{-(t-\overline{\tau})}(m)\right)<\epsilon.
\end{eqnarray}
By hypothesis \textbf{A3} it is possible to define $A=\inf{\left\{e^{-\int_{0}^{t}(u'(m)dt}\right\}}>0$, moreover by recalling the hypothesis \textbf{B2} and the preliminary result \eqref{prelP}
\begin{eqnarray}
\label{1}
e^{-\int_{0}^{t}(u'(m)-\alpha(P_{n}(m,s),m))ds}&\leq& e^{-At} e^{\int_{0}^{\overline{\tau}}|\alpha(P_{n}(m,s),m)|ds}\nonumber\\
&\leq& A e^{k_{a} |P_{n}(m,s)|\overline{\tau}} \leq A e^{k_{a} \epsilon \overline{\tau}}.
 \end{eqnarray}
Furthermore, in order to estimate $|F_{C_n}|$, we note that by hypotheses \textbf{A4} and \eqref{A5}, with $c$ as in \eqref{solc} and computed in $C_{n}$
\begin{eqnarray*} 
\Bigg|\sigma(m)\int_{\underline{\tau}}^{\overline{\tau}}c\log(c)da \Bigg|<<1.
\end{eqnarray*}
Moreover by hypothesis \eqref{A5} and $\bf{B2}$ and by using \eqref{prelP}
\begin{eqnarray*}
\Bigg|e^{\frac{\widetilde{\alpha}(P_n(m,t),m)}{\sigma(m)}}e^{-\frac{\widetilde{\alpha}(P_n(t,m),m)}{\sigma(m)}e^{-\sigma(m)t}}e^{\ln(\Omega(a,m))e^{-\sigma(m)t}}\Bigg|\\
\leq \Big|e^{\frac{\widetilde{\alpha}(P_n(m,t),m)}{\sigma(m)}}\Big|\Big|e^{-\frac{\widetilde{\alpha}(P_n(t,m),m)}{\sigma(m)}}\Big|\Big|e^{\ln(\Omega(a,m))}\Big|\leq \Big|e^{\ln(\Omega(a,m))}\Big| <1.
\end{eqnarray*}
So we recover the estimate 
\begin{eqnarray}
|F_{C_n}|<1
\end{eqnarray}
and consequently by using again \eqref{prelP} we have
\begin{eqnarray}
\label{2}
F((P_{n}(m,t),m)\leq k\epsilon-1\leq k\epsilon.
\end{eqnarray}
As a consequence  we get the  estimate 
\begin{eqnarray}
\label{FP}
\Big|\int_{0}^{m}F(P_{n}(s,z),z)dz\Big|&\leq& \int_{0}^{g(1)}|F(P_{n}(s,z),z)|dz \nonumber\\
&\leq& g(1)|F(P_{n}(s,z),z)|\leq g(1)k\epsilon
\end{eqnarray}
Finally, by \eqref{1}, \eqref{2} and \eqref{FP} we obtain for any $n\geq 0$
\begin{eqnarray}
|C_n(m,t)|&\leq& \epsilon A e^{k_{a} \epsilon \overline{\tau}}+ \int_{0}^{\overline{\tau}}A e^{k_{a} \epsilon\overline{\tau} }g(1)k\epsilon dt\nonumber \\
&\leq&\epsilon A e^{k_{a} \epsilon \overline{\tau}} (1+g(1)\overline{\tau})\\
&\leq& \overline{K} \epsilon.
\end{eqnarray}
Let us now prove the convergence of the sequence $C_{n}$ by estimating
\begin{eqnarray}
\label{57}
|C_{n+1}\!\!\!\!&-&\!\!\!\!C_{n}|\leq e^{\int_{0}^{t}u'(m)dt}\Omega(a,m)\Big| e^{\int_{0}^{t}\alpha(P_{n+1}(m,s),m)ds}-e^{\int_{0}^{t}\alpha(P_n(m,s),m)ds}\Big| \nonumber \\
\!\!\!\!&+&\!\!\!\! \Big|\int_{0}^{t}\!\!\!e^{\int_{s}^{t}u'(m)dr}e^{-\int_{t}^{s}\alpha(P_{n+1}(m,r),m)dr}\!\!\!\int_{0}^{m}\!\!\!\!\!\!F(P_{n+1}(s,z),z)dzds \nonumber\\ 
\!\!\!\!&-&\!\!\!\! \int_{0}^{t}\!\!\!e^{\int_{s}^{t}u'(m)dr}e^{-\int_{t}^{s}\alpha(P_{n}(m,r),m)dr}\!\!\!\int_{0}^{m}\!\!\!\!\!\!F(P_{n}(s,z),z)dzds\Big| \nonumber\\
\!\!\!\!&=&\!\!\!\! I_1+ I_2
\end{eqnarray}
Recalling definitions of $\pi$ and $g$  and hypothesis \textbf{A3} and \textbf{A4} we have that\\
 $\pi_{-(t-\overline{\tau})}(m)<m<g(1)$ so
\begin{eqnarray}
e^{-u(m)t}\Omega\left(\overline{\tau},\pi_{-(t-\overline{\tau})}(m)\right)<\epsilon.
\end{eqnarray}
By hypothesis \textbf{B2} we can estimate $I_1$ as follows 
\begin{eqnarray}
\label{firstrow}
\Big|\!\!\!\!\!\!&e&\!\!\!\!\!\!^{\int_{0}^{t}\alpha(P_{n+1}(m,s),m)-\alpha(P_n(m,s),m)ds}\Big|\Big|e^{\int_{0}^{t}\alpha(P_n(m,s),m)ds}\Big| \nonumber\\
&+&\Big| e^{\int_{0}^{t}\alpha(P_{n}(m,s),m)-\alpha(P_{n+1}(m,s),m)ds}\Big|\Big|e^{\int_{0}^{t}\alpha(P_{n+1}(m,s),m)ds}\Big|\nonumber \\
&\leq& e^{\int_{0}^{t}|P_n(m,s)|ds}e^{\int_{0}^{t}|P_{n+1}(m,s)-P_{n}(m,s)|ds} \nonumber \\
&+&e^{\int_{0}^{t}|P_{n+1}(m,s)|ds}e^{\int_{0}^{t}|P_{n}(m,s)-P_{n+1}(m,s)|ds} \nonumber \\
&\leq& 2e^{\int_{0}^{t}|P_n(m,s)|ds} \int_{0}^{t}|P_{n+1}(m,s)-P_n(m,s)|ds  \nonumber \\
&\leq&  2e^{\tau|P_n(m,s)|}\sup{|P_{n+1}(m,s)-P_n(m,s)|}.
\end{eqnarray}
Let us now focus on $I_2$. By adding and subtracting the term
\begin{displaymath}
\int_{0}^{t}\!\!\!e^{\int_{s}^{t}u'(m)dr}e^{-\int_{t}^{s}\alpha(P_{n+1}(m,r),m)dr}\!\!\!\int_{0}^{m}\!\!\!\!\!\!F(P_{n}(s,z),z)dzds
\end{displaymath} 
and defining $\widetilde{A}=\sup{\left\{e^{\int_{s}^{t}u'(m)dr}\right\}}>0$ we obtain the following estimate
\begin{eqnarray}
\label{59}
|C_{n+1}\!\!\!\!&-&\!\!\!\!C_{n}| \nonumber\\
\!\!\!\!&\leq&\!\!\!\! 2e^{\tau|P_n(m,s)|}\sup{|P_{n+1}(m,s)-P_n(m,s)|} \nonumber\\
\!\!\!\!&+&\!\!\!\!\int_{0}^{t}\!\!\!\widetilde{A} \Big| e^{-\int_{t}^{s}\alpha(P_{n+1}(m,r),m)dr}\Big| \nonumber\\
\!\!\!&\times&\!\!\!\int_{0}^{g(1)}\!\!\!\!\!\!\big|F(P_{n+1}(s,z),z)-F(P_{n}(s,z),z)\big|dzds\nonumber\\
\!\!\!\!&+&\!\!\!\!\int_{0}^{t}\!\!\! \widetilde{A} \Big|e^{-\int_{t}^{s}\alpha(P_{n+1}(m,r),m)dr}-e^{-\int_{t}^{s}\alpha(P_{n}(m,r),m)dr}\Big|\nonumber\\
\!\!\!&\times&\!\!\!\int_{0}^{g(1)}\!\!\!\!\!\!\big|F(P_{n}(s,z),z)\big|dzds
\end{eqnarray}
This estimates involves computations perfectly analogous to the ones for $I_1$ except for the the term $\Big|F(P_{n+1}(s,z),z)-F(P_{n}(s,z),z)\Big|$.\\
 By using the hypothesis \textbf{B2} we have that
\begin{eqnarray}
\label{FP1}
\Big|F(P_{n+1}(s,z),z)\!\!\!\!\!&-&\!\!\!\!\!F(P_{n}(s,z),z)\Big| \nonumber\\
&\leq&\Big|\alpha(P_{n+1}(m,t),m)P_{n+1}(m,t)-\alpha(P_{n}(m,t),m)P_{n}(m,t)\Big|\nonumber\\
&+&\Big|F_{C_{n+1}}-F_{C_{n}}\Big|\nonumber\\
\!\!\!\!\!&\leq&\!\!\!\!\! k_a \Big|P_{n+1}(m,t)-P_{n}(m,t)\Big|+\Big|F_{C_{n+1}}-F_{C_{n}}\Big|.
\end{eqnarray}
We now estimate separately the term $\Big|F_{C_{n+1}}-F_{C_{n}}\Big|$ in the following way
\begin{eqnarray*}
\Big|F_{C_{n+1}}\!\!\!\!\!&-&\!\!\!\!\!F_{C_{n}} \Big|\\
\!\!\!\!\!&\leq&\!\!\!\!\!\Bigg|\exp\!\!\left[\frac{\widetilde{\alpha}(P_{n+1}(m,t),m)}{\sigma(m)}\!\!-\!\!\left(\frac{\widetilde{\alpha}(P_{n+1}(t,m),m)}{\sigma(m)}\!-\! \ln(\Omega(a,m))\!\!\right)e^{-\sigma(m)t}\!\right]\\
\!\!\!\!\!&-&\!\!\!\!\!\exp\!\!\!\ \left[\frac{\widetilde{\alpha}(P_{n}(m,t),m)}{\sigma(m)}-\left(\frac{\widetilde{\alpha}(P_{n}(t,m),m)}{\sigma(m)}-\ln(\Omega(a,m))\right)e^{-\sigma(m)t}\right] \Bigg|.
\end{eqnarray*}
This can be rewritten as
\begin{eqnarray}
\label{fcresult}
\Big|F_{C_{n+1}}-F_{C_{n}} \Big|&\leq& \Bigg| e^{\frac{\widetilde{\alpha}(P_{n+1}(m,t),m)}{\sigma(m)}}- e^{-\frac{\widetilde{\alpha}(P_{n+1}(m,t),m)}{\sigma(m)}e^{-\sigma(m)t}}\nonumber\\
&-& \left(e^{\frac{\widetilde{\alpha}(P_{n+1}(m,t),m)}{\sigma(m)}}-e^{-\frac{\widetilde{\alpha}(P_{n}(m,t),m)}{\sigma(m)}e^{-\sigma(m)t}}\right) \Bigg|.
\end{eqnarray}

Now, if we sum and subtract in \eqref{fcresult} the quantity
\begin{displaymath}
 e^{\frac{\widetilde{\alpha}(P_{n}(m,t),m)}{\sigma(m)}}e^{-\frac{\widetilde{\alpha}(P_{n+1}(m,t),m)}{\sigma(m)}e^{-\sigma(m)t}}
\end{displaymath}
we obtain
\begin{eqnarray*}
\Big|F_{C_{n+1}}-F_{C_{n}}\Big| \!\!\!\!&\leq& \!\!\!\! \Bigg| e^{\frac{\widetilde{\alpha}(P_{n+1}(m,t),m)}{\sigma(m)}} -e^{\frac{\widetilde{\alpha}(P_{n}(m,t),m)}{\sigma(m)}}\Bigg| \Bigg| e^{-\frac{\widetilde{\alpha}(P_{n+1}(m,t),m)}{\sigma(m)}e^{-\sigma(m)t}}\Bigg| \\
\!\!\!\!&+&\!\!\!\!  \Bigg| e^{-\frac{\widetilde{\alpha}(P_{n+1}(m,t),m)}{\sigma(m)}e^{-\sigma(m)t}}- e^{-\frac{\widetilde{\alpha}(P_{n}(m,t),m)}{\sigma(m)}e^{-\sigma(m)t}}\Bigg|\Bigg|e^{\frac{\widetilde{\alpha}(P_{n}(m,t),m)}{\sigma(m)}}\Bigg|.
\end{eqnarray*} 
So we can estimate
 \begin{eqnarray}
\label{fcfin}
\Big|F_{C_{n+1}}-F_{C_{n}}\Big|\!\!\!&\leq&\!\!\! \Bigg|  e^{\frac{\widetilde{\alpha}(P_{n+1}(m,t),m)}{\sigma(m)}} -e^{\frac{\widetilde{\alpha}(P_{n}(m,t),m)}{\sigma(m)}}  \Bigg|M \nonumber \\
\!\!\!&+& \!\!\! \Bigg| e^{-\frac{\widetilde{\alpha}(P_{n+1}(m,t),m)}{\sigma(m)}e^{-\sigma(m)t}}- e^{-\frac{\widetilde{\alpha}(P_{n}(m,t),m)}{\sigma(m)}e^{-\sigma(m)t}}\Bigg|\overline{M}\nonumber\\
\!\!\!&\leq&\!\!\! A M e^{k_a(t-r)}\Big|\frac{\alpha(P_{n+1}(m,t),m)}{\sigma(m)}- \frac{\alpha(P_{n}(m,t),m)}{\sigma(m)}\Big|\nonumber\\
\!\!\!&+&\!\!\! A \overline{M}e^{k_a(t-r)}\Big|\frac{\alpha(P_{n+1}(m,t),m)}{\sigma(m)}e^{-\sigma(m)t}\nonumber\\
\!\!\!&-&\!\!\!\frac{\alpha(P_{n}(m,t),m)}{\sigma(m)}e^{-\sigma(m)t} \Big|\nonumber\\
\!\!\!&\leq&\!\!\! A Me^{k_a \overline{\tau}} k_a |P_{n+1}(m,t)-P_{n}(m,t)|\nonumber\\
\!\!\!&+&\!\!\!A \overline{M} e^{k_a \overline{\tau}} k_a e^{-\sigma(m)t}|P_{n+1}(m,t)-P_{n}(m,t)|.
\end{eqnarray}

By collecting our estimates \eqref{fcfin}, \eqref{firstrow}, \eqref{FP}, \eqref{FP1} we have
\begin{eqnarray}
\label{final}
|C_{n+1}-C_{n}|\!\!\!\!\!&\leq&\!\!\!\!\! 2\epsilon e^{\tau|P_n(m,s)|} \sup{|P_{n+1}(m,s)-P_n(m,s)|} \nonumber \\
\!\!\!\!\!&+&\!\!\!\!\!\widetilde{A}e^{\tau|P_n(m,s)|}k_a \big|P_{n+1}(m,t)-P_{n}(m,t)\big|\nonumber\\
\!\!\!\!\!&+&\!\!\!\!\! A Me^{k_a \overline{\tau}}k_a  \big|P_{n+1}(m,t)-P_{n}(m,t)\nonumber\big|\\
\!\!\!\!\!&+&\!\!\!\!\! A \overline{M} e^{k_a \overline{\tau}}k_a e^{-\sigma(m)t}|P_{n+1}(m,t)-P_{n}(m,t)\big|\nonumber\\
\!\!\!\!\!&-&\!\!\!\!\!2\epsilon k \tau g(1) e^{\tau|P_n(m,s)|} \sup{|P_{n+1}(m,s)-P_n(m,s)|}
\end{eqnarray}
Finally from \eqref{final} it follows the global stability for the $C-cells$ that is
\begin{eqnarray}
|C_{n+1}-C_{n}|\rightarrow 0 \quad \textrm{as $n\rightarrow \infty$}.
\end{eqnarray}
Hence we proved the following Theorem.
\begin{teo}[Global exponential stability]
\label{T4}
If the hypotheses \eqref{A5}, $\bf{B1}$, ${\bf B2}$, ${\bf A1}$ - ${\bf A4}$,  hold, with $k_b$ such that $\frac{k_b(1+2\varsigma)}{I}<1$, then the trivial solution $C\equiv 0$ of \eqref{eqC}  is globally exponentially stable in the sense of definition \bf{Def 3}.
\end{teo}
In conclusion we can summarize  the results obtained  in Theorems \ref{T2}, \ref{T3}, \ref{T4} as follows
\begin{teo}[Global exponential stability for the total system]
If the hypotheses $\bf{B1}$, ${\bf B2}$, ${\bf A1}$ - ${\bf A4}$ holds, with $k_b$ such that $\frac{k_b(1+2\varsigma)}{I}<1$, then, the trivial solution $(N,P,C)\equiv (0,0,0)$, is globally exponentially stable in the sense of definition {\bf Def 3} for the problem \eqref{eqN}, \eqref{Psol}, \eqref{eqC}.
\end{teo}
        
\section{Numerical simulations}

In order to confirm the stability results obtained analytically, we performed some numerical simulations with suitable life-parameters on the discretization of the equations \eqref{Nsol}, \eqref{Psol}, \eqref{Csol} with stationary solutions defined as $N^{*}$, $P^{*}$, $C^{*}$.\\
To approximate the integral terms of \eqref{Nsol}, \eqref{Psol}, \eqref{Csol} that here we denote as $A_1$, $A_2$ and $A_3$, we used a composite trapezoidal quadrature formula. Note that $A_1$, $A_2$, $A_3$ turn out to be block matrices whose dimension depends on the discretization step.

As first step we made use of a Picard iteration for the solution of the linear fixed point problem arising from the discretization of \eqref{Nsol} given by
\begin{eqnarray*}
N^{k+1}=A_{1}N^{k}+b_{1}. 
\end{eqnarray*}
By using an error estimate $||N^{k}-N^{k-1}||<toll$, we denote by $N^{k}\approx N^{*}$ the $k-th$ final iterate (depending naturally on the fixed tolerance). Then we insert $N^{k}$ into the linear fixed point problem arising from the discretization of \eqref{Psol} and apply the second Picard iteration
\begin{eqnarray*}
P^{k+1}=A_{2}P^{k}+b_{2}(N^{k}).
\end{eqnarray*}
With the same procedure, once obtainted $P^{k}\approx P^{*}$ such that is satisfied the error estimate $||P^{k}-P^{k-1}||<toll$, we insert finally $N^{k}$ and $P^{k}$ into the linear fixed point problem arising from the discretization of \eqref{Csol}
\begin{eqnarray*}
C^{k+1}=A_{3}C^{k}+b_{2}(N^{k},P^{k}).
\end{eqnarray*}
In order to illustrate the convergence behaviour we plot in Figure $1$ and $2$:

\begin{itemize}
	\item the iteration error defined for a generic $F$ as $e^{k}=||F^{k}-F^{*}||$;
	\item the asymptotic rate of convergence estimated as $r^{k}=\frac{||e^{k}||}{||e^{k-1}||}$
\end{itemize}
 
The obtained results confirm the stability properties proved analytically in Section $5$, and also observed from the biological point of view: the stability of the stem cell population in a physiological tissue stabilize the global homeostasis of the environment.

As future work we plan to investigate numerically more sophisticated real-life models.

 
 \begin{figure}[!h]
 \centering
   \includegraphics[width=5cm]{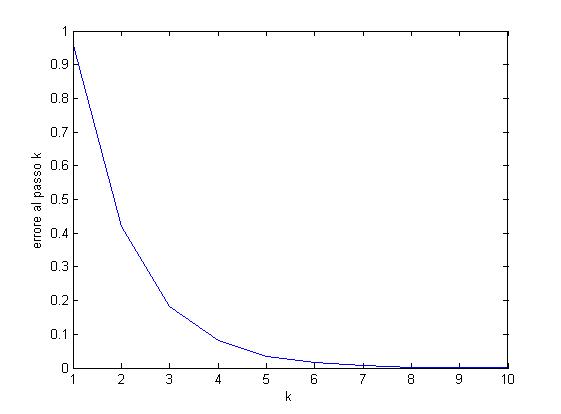}
  \quad 
   \includegraphics[width=5cm]{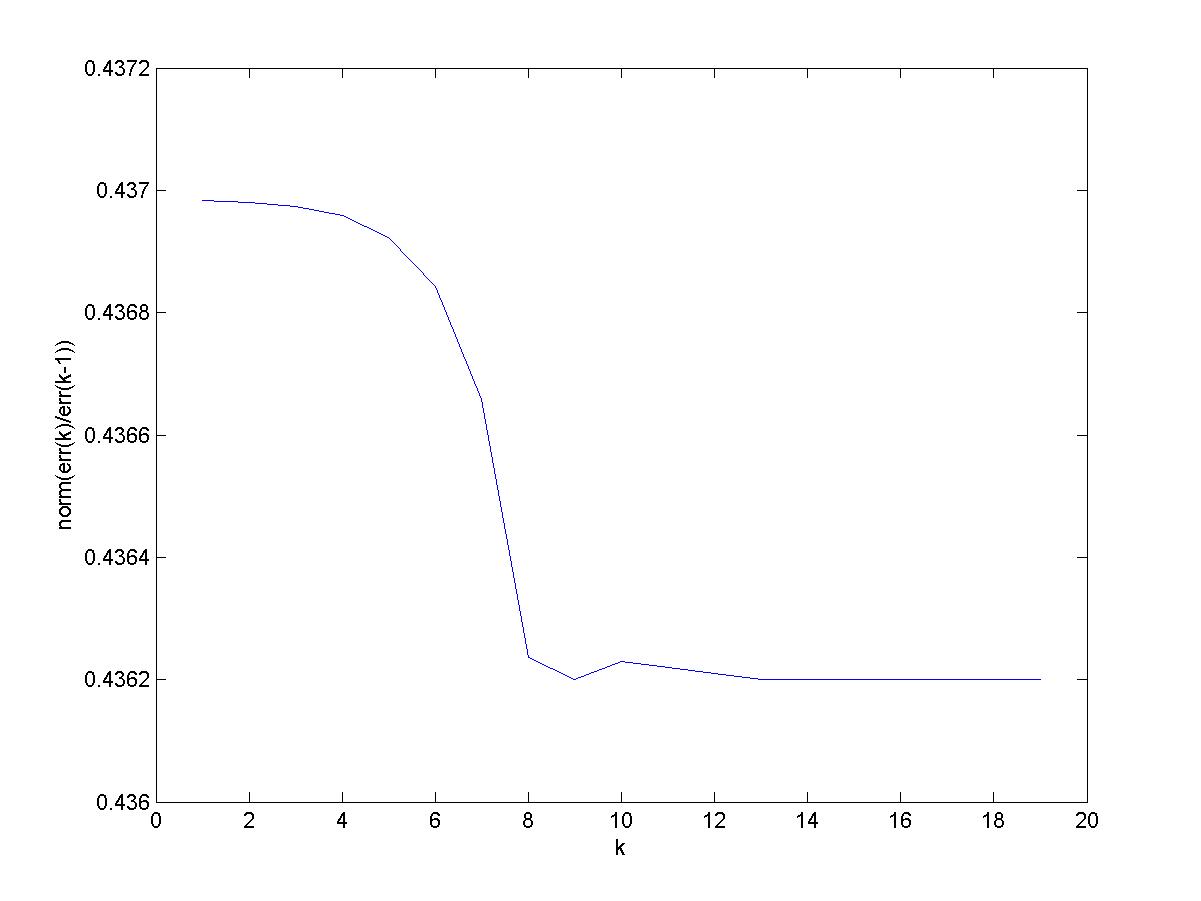}
 \caption{Error and rate of convergence of the Picard iteration on $N(m,t)$}
 \end{figure}
 \begin{figure}[!h]
   \includegraphics[width=5cm]{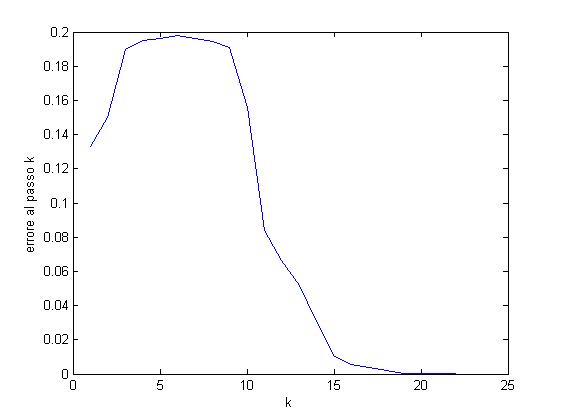}
   \quad
   \includegraphics[width=5cm]{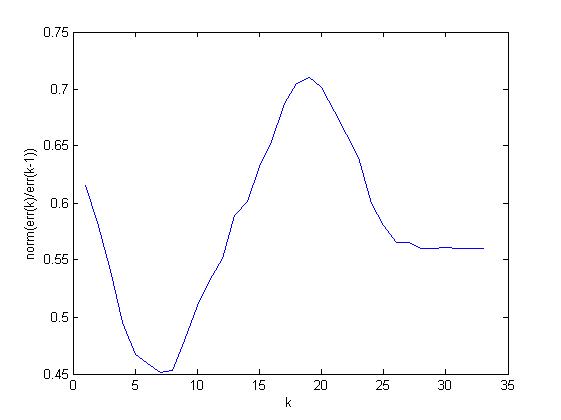}
 \caption{Error and rate of convergence of the Picard iteration on $C(m,t)$}
 \end{figure}

\section{Conclusions and future work}
We setup a model able to reproduce, the experimental validated evidence, that a good presence of stem cells in a tissue environment improve critically the activity of self-regeneration and repair of the genetic damages that could accumulate physiologically and modify the healthy cell life cycle causing a wide range of diseases, in particular in this work we considered a kind damage that cause iper-activity of cells at low densities and make them more aggressive.\\
This model was inspired by reading about the new trend of medical research that is attempting to power in-vitrio this capability of stem cells in order to make possible the therapy of various types of diseases by means of a  transplantation of a critical density of stem cells physically in the damaged areas.\\
In this context, this can be proposed, once computerized, as first rudimental tool to complement the work of physicians and biologists to test the appropriate concentrations and the timing of the reaction.\\
Our future work, as well as computerization, is aimed to test the validity and robustness of our model through its application to other types of cellular damages.

\section*{Acknowlegments}
The authors would like to thank {\em Prof. Nicola Guglielmi} (Department of Information Engineering, Computer Science and Mathematics, University of L'Aquila) for his support and helpful suggestions and discussions in setting up the numerical simulations in this paper.


\begin{thebibliography}{100}


\bibitem{rif10} 
M.~Adimy, F. ~Crauste, \emph{Global stability of a partial differential equation with distributer delay due to cellular replication.}, Nonlinear analysis, {\bf 54},  1469--1491, 2003.

\bibitem{rif8} 
M.~Adimy, L.P.~Menjouet,  \emph{Asymptotic behaviour of a singular transport equation modelling cell division}, Discrete and continuous dynamical systems-series B, {\bf 3}, no. 3, 439--456, 2003.


\bibitem{rif7} 
S.~Ani{\c{t}}a, M.~Iannelli, M.-Y.~Kim,   E.-J.~ Park, \emph{Optimal harvesting for periodic age dependent population dynamics.}, Siam J. of Appl. Mat. {\bf 58}, no.5, 1648--1666, 1998.

\bibitem{rif3} 
M.~Gurtin, R.C.~MacCamy, \emph{Non-linear age dependent population dynamics}, Arch. Ration. Mech. Analysis {\bf 54}, 281--300, 1974	

\bibitem{rif2} 
F.~Hoppensteadt, \emph{Mathematical Theories of populations: Demographics, Genetics and Epidemics}, Society  for Industrial and Applied Mathematics,1975.	

\bibitem{rif4} 
M. ~Iannelli, \emph{Mathematical Theory of Age-Structured Population Dynamics}, Giardini editori, Pisa 1994.	
 
\bibitem{rif12} 
H.H. Lloyd \emph{Estimation of tumor cell kill from Gompertz growth curves}, Europe PubMed Central, {\bf 59}, 267--277, 1975. 
 
\bibitem{rif9} 
M. ~Mackey, R.~Rudnicky, \emph{Global stability in a delayed partial differential equation describing cellular replication}, J. Math. Biology, {\bf 33}, 89--109, 1994.

\bibitem{rif11} 
M.~Mackey, R.~Rudnicky, \emph{A new criterion for the global stability of simultaneous cell replication and maturation processes}, J. Math. Biology, {\bf 38},  195--219, 1999

 \bibitem{rif5} 
 P.~Marcati, \emph{On the Global Stability of the Logistic Age-Dependent Population Growth},  J. Math. Biology,  {\bf 15}, 215--226, 1982.

\bibitem{rif6} 
P. Marcati, \emph{Some considerations on the mathematical approach to nonlinear age dependent population dynamics}, Comput. Math. Appl. {\bf 9} , no. 3, 361--370, 1983. 
  
\bibitem{rif1} 
G.F.~Webb, \emph{Theory of nonlinear age-dependent population dynamics}, Marcel Dekker inc.,1985.



\end{thebibliography}
\end{document}